\documentclass[11pt]{35}
\usepackage{amsthm, amsmath, amssymb, amsbsy, amsfonts, latexsym, euscript}
\usepackage{calrsfs} 
\usepackage{graphicx}

\DeclareMathAlphabet{\varmathbb}{U}{pxsyb}{m}{n}

\def\leq{\leqslant}

\def\geq{\geqslant}
\def\phi{\varphi}

\def\kappa{\varkappa}

\newcommand{\D}{\mathrm{d}\kern0.2pt}%

\newcommand{\E}{\mathrm{e}\kern0.2pt} 
\newcommand{\ii}{\kern0.05em\mathrm{i}\kern0.05em}

\newcommand{\RR}{\mathbb{R}}%
\newcommand{\CC}{\mathbb{C}}%

\newtheorem{theorem}{\bf \indent Theorem}[section]
\newtheorem{proposition}{\bf \indent Proposition}[section]

\newtheorem{corollary}{\bf \indent Corollary}[section]

\theoremstyle{remark} 
\newtheorem{remark}{\bf \indent Remark}[section]

\newtheorem{definition}{\bf\indent Definition}[section] 

\newtheorem{conjecture}{\bf \indent Conjecture}[section]

\numberwithin{equation}{section}

\begin{document}

\renewcommand{\thefootnote}{}
\renewcommand{\footnoterule}{}

\vskip5mm

\noindent {\Large \bf MEAN VALUE PROPERTIES OF HARMONIC FUNCTIONS \\[3pt] AND
RELATED TOPICS (A SURVEY)}

\vskip4mm

{\bf N. Kuznetsov}

\vskip-2pt {\small Laboratory for Mathematical Modelling of Wave Phenomena}
\vskip-4pt {\small Institute for Problems in Mechanical Engineering, RAS} \vskip-4pt
{\small 61, V.O., Bol'shoy pr., St. Petersburg 199178, Russia} \vskip-4pt {\small
nikolay.g.kuznetsov@gmail.com}

\vskip4mm

\parbox{146mm} {\noindent Results involving various mean value properties are
reviewed for harmonic, biharmonic and metaharmonic functions. It is also considered
how the standard mean value property can be weakened to imply harmonicity and
belonging to other classes of functions.}

\vskip7pt

{\centering
\section{Introduction}
}

\noindent It is highly likely that the mean value theorem is the most remarkable and
useful fact about harmonic functions. Results on this and other properties of
harmonic functions were surveyed by Netuka \cite{N} in the remote 1975, but,
unfortunately, there is a number of inaccuracies in this paper. To the best
author's knowledge, only one review in this area had appeared since then; namely,
the extensive article \cite{NV} by Netuka and Vesel\'y which is a substantially
extended version of Netuka's survey updated to mid-1993, but still reproducing some
of the inaccuracies from the previous paper.

During the past 25 years (\cite{NV} was published in 1994), rather many papers on
mean value properties and other related topics have appeared. Some of these contain
results of significant interest; see, for example, \cite{HN,HN1,FM} to list a few. A
number of new as well as some old results deserve to be reviewed, especially,
various forms of converse mean value theorem. Mean value theorems are also
considered for solutions of equations different from the Laplace equation. Moreover,
for several old results, for which only two-dimensional versions were published,
proofs are provided for general formulations or just outlined because they were not
properly presented in \cite{NV}. However, the so-called inverse mean value
properties (see \cite{NV}, Sections~7 and 8) are not treated here, in particular,
because these properties were considered in detail in Zaru's thesis \cite{LZ}
available online. Sections~2 and 3 of her thesis deal with inverse mean value
properties on balls and annuli and strips, respectively. Of course, many references
to the paper \cite{NV} are given, since results reviewed here continue research
initiated before 1993 and described in Netuka and Vesel\'y's article.

In the remaining part of this section, basic classical results about harmonic
functions are presented as the basis for considerations in Sections 2--4, and so
bibliography used here is restricted to a few monographs, textbooks and pioneering
papers; see \cite{N} and \cite{NV} for further references. We begin with the
standard formulation of the mean value theorem (see, for example, the monograph
\cite{GT} by Gilbarg and Trudinger or the textbook \cite{M} by Mikhlin).

\begin{theorem}
Let $D$ be a domain in $\RR^m$, $m \geq 2$. If $u \in C^2 (D)$ satisfies the Laplace
equation $\nabla^2 u = 0$ in $D$, then we have
\begin{equation}
u (x) = \frac{1}{m \, \omega_m r^{m-1}} \int_{\partial B} u \, \D S \label{A}
\end{equation}
for every ball $B = B_r (x)$ such that $\overline B \subset D$.
\end{theorem}

\noindent Here and below the following notation is used:

$\RR^m$ is the Euclidean $n$-space with points $x = (x_1, \dots , x_m)$, $y = (y_1,
\dots , y_m)$, where $x_i, y_i$ are real numbers, whereas the norm is $|x| = (x_1^2
+ \dots + x_m^2)^{1/2}$.

For a set $G \subset \RR^m$, by $\partial G$ we denote its boundary, whereas
$\overline G = G \cup \partial G$ is its closure; $D \subset \RR^m$ is a domain if
it is an open and connected set, not necessarily bounded. In particular, $B_r (x) =
\{ y : |y-x| < r \}$ denotes the open ball of radius $r$ centred at $x$, the volume
of unit ball is $\omega_m = 2 \pi^{m/2} / [m \Gamma (m/2)]$ and $\D S$ is the
surface area measure.

By $\nabla = (\partial_1, \dots , \partial_m)$ the gradient operator is denoted;
here $\partial_i = \partial / \partial x_i$, $\partial_i \partial_j = \partial^2 /
\partial x_i \partial x_j$ etc. Finally, $C^k (D)$ is the set of continuous
functions in $D$, whose derivatives of order $\leq k$ are also continuous there;
functions in $C^k (\overline D)$ are continuous in $\overline D$ with all
derivatives of order $\leq k$.

Since the denominator in \eqref{A} is equal to the area of sphere of radius $r$, it
is common to refer to this equality as the area version of mean value theorem. Also,
it is known as Gauss' theorem of the arithmetic mean; see \cite{K}, p.~223. Indeed,
one finds this theorem in his paper \emph{Algemeine Lehrsatze in Beziehung auf die
im verkehrtem Verhaltnisse des Quadrats der Entfernung Wirkenden Anziehungs- und
Abstos\-sungs-Krafte} published in 1840 (see also \emph{Gauss Werke}, Bd. 5.
S.~197--242); the corresponding quotation from this paper is given in \cite{NV},
p.~361.

The following consequence of Theorem~1.1 is not widely known, but has important
applications in the linear theory of water waves; see the monograph \cite{LWW},
Sect.~4.1, where a proof of this assertion is given, and the article \cite{KM},
where further references can be found.

\begin{corollary}
Zeros of a harmonic function are never isolated.
\end{corollary}

Harmonic functions also have a mean value property with respect to volume measure;
see, for example, \cite{H}, p.~12. It follows by integrating \eqref{A} with respect
to the polar radius at $x$ from 0 to some $R > 0$, thus yielding.

\begin{theorem}
Let $D$ be a bounded domain in $\RR^m$, $m \geq 2$, and let $B_R (y) \subset D$ be
any open ball such that $R$ is less than or equal to the distance from $y$ to
$\partial D$. If $u$ is a Lebesgue integrable harmonic function in $D$, then we
have
\begin{equation}
u (y) = \frac{1}{\omega_m R^{m}} \int_{B_R (y)} u \, \D x \, , \label{V}
\end{equation}
where the ball's volume stands in the denominator.
\end{theorem}

\noindent In what follows, we denote the right-hand side terms in \eqref{A} and
\eqref{V} by $L (u,x,r)$ and $A (u,x,r)$, respectively (the notation used in
\cite{NV}).

There are many corollaries of Theorem 1.2 and the most important is the maximum
principle (also referred to as the strong maximum principle; see, for example,
\cite{GT}, p.~15).

\begin{theorem}
Let $D$ be a bounded domain. If a harmonic function $u$ attains either a global
minimum or maximum in $D$, then $u$ is constant.
\end{theorem}

\noindent This theorem implies several global estimates (see \cite{GT}, Ch.~2);
first we formulate those known as the weak maximum principle.

\begin{corollary}
Let $u \in C^2 (D) \cap C^0 (\overline D)$, where $D$ is a bounded domain. If $u$ is
harmonic in $D$, then
\[ \min_{x \in \partial D} u \leq u (x) \leq \max_{x \in \partial D} u \quad for \ all 
\ x \in \overline D ,
\]
where equalities hold only for $u \equiv \mathrm{const}$.
\end{corollary}

\noindent Moreover, the derivatives of a harmonic function are estimated in terms
of the function itself.

\begin{corollary}
Let $u$ be harmonic in $D$. If $D'$ is a compact subset of $D$, then
\[ \max_{x \in D'} |\partial_i u| \leq \frac{m}{d} \sup_{x \in D} |u| \quad for 
\ i=1,\dots,m ,
\]
where $d$ is the distance between $D'$ and $\partial D$.
\end{corollary}

In his note \cite{Ko} published in 1906, Koebe announced the following converse of
Theorem~1.1.

\begin{theorem}
Let $D$ be a bounded domain in $\RR^m$, $m \geq 2$. If $u \in C^0 (\overline D)$
satisfies equality \eqref{A} for every ball $B = B_r (x)$ such that $\overline B
\subset D$, then $u$ is harmonic in~$D$.
\end{theorem}

\noindent There is a slightly stronger version of this theorem requiring that
\eqref{A} holds not at every $x \in D$, but only for some sequence $r_k (x) \to 0$
as $k \to \infty$. Of several proofs of this theorem, we mention two. In his
classical book \cite{K}, pp.~224--226, Kellogg used straightforward but cumbersome
calculations for establishing that $u \in C^2 (D)$. In the Mikhlin's textbook
\cite{M}, these calculations are replaced by application of mollifier technique (see
the proof of Theorem~11.7.2 in \cite{M}). Moreover, the latter book contains the
following important consequence of Corollary~1.2 and Theorem~1.4 (see \cite{M},
Theorem~11.9.1).

\begin{corollary}[Harnack's convergence theorem]
Let a sequence $\{u_k\}$ consist~of $C^0 (\overline D)$-functions harmonic in $D$.
If $u_k \to u$ uniformly on $\partial D$ as $k \to \infty$, then $u_k \to u$
uniformly on $\overline D$ and $u$ is harmonic in $D$.
\end{corollary}

\noindent Various other results about convergence of harmonic functions can be found
in Brelot's lectures \cite{Br} (Appendix, Sect.~19) along with an elegant proof of
Theorem~1.4 based on Poisson's formula and the maximum principle (Appendix,
Sect.~18).

It is worth mentioning that two early versions of converse to Theorem~1.1 were
published by different authors under the same title; see \cite{L} and \cite{T}. In
the first of these, E. E. Levi independently proved the two-dimensional version of
Theorem~1.4. (However, this mathematician is more widely known for his paper of
1907, in which a fundamental solution to a general elliptic equation of second order
with variable coefficients is constructed; see the monograph \cite{Mi} by Miranda.)
In the second paper, Tonelli relaxed the assumptions imposed by Levi; namely, $u$ is
required to be Lebesgue integrable on $D$. Further historical remarks and
characterizations of harmonic functions analogous to Theorem~1.4, but expressed in
terms of equality \eqref{V} instead of \eqref{A}, can be found in \cite{NV},
pp.~363--364.

In view of Theorem 1.4, the validity of equality \eqref{A} for every ball $B = B_r
(x)$ such that $\overline B \subset D$ can be taken as the definition of harmonicity
for functions locally integrable in a domain $D$. Another illustration (due to
Uspenskii \cite{U}) that this definition is reasonable is not so well-known. In this
two-dimensional considerations, the circumference centred at $x$ is denoted by $C$
instead of $\partial B$ and $L (u,x,C)$ stands for the mean value of $u$ over $C$.

First, we notice that $L (u,x,C)$ has the following properties:

(i) $L (u,x,C) \geq 0$ provided $u \geq 0$ on $C$;

(ii) $L (u_1 + u_2,x,C) = L(u_1,x,C) + L (u_2,x,C)$ for $u_1$ and $u_2$ given on
$C$;

(iii) $L (\alpha u,x,C) = \alpha L (u,x,C)$ for all $\alpha \in \RR$;

(iv) $L (u,x,C) = 1$ provided $u \equiv 1$ on $C$;

(v) $L (u,x,C) = L (u_*,x,C_*)$ when two circumferences $C$ and $C_*$ ($u$ and $u_*$
are given on the respective curve) are congruent in such a way that the functions'
values are equal at the corresponding points.

\noindent Some of these relations are obvious and the others are straightforward to
verify.

It occurs that (i)--(v) provide an axiomatic definition of the mean value over $C$
for the class of integrable functions. Indeed, such a definition is equivalent to
\eqref{A} and all facts that follow from the latter formula can be proved on the
basis of (i)--(v). In particular, the main result of \cite{U} is derivation of
Corollary~2 from these relations in the case of a disc. Also, an interesting
representation in geometric terms is found for the function solving the Dirichlet
problem in a disc when a step function is given on its boundary.

Of course, the assertion analogous to Theorem 1.4 with formula \eqref{V} replacing
\eqref{A} is true as well as some weaker formulations; see \cite{ABR}, pp.~17--18
and \cite{NV},~p.~364.

Another obvious consequence of harmonicity is the following.

\begin{proposition}[Zero flux property]
Let $u$ be harmonic in a domain $D \subset \RR^m$, $m \geq 2$. If $D'$ is a bounded
subdomain such that $\overline{D'} \subset D$ and $\partial D'$ is piecewise smooth,
then\\[-3mm]
\begin{equation}
\int_{\partial D'} \frac{\partial u}{\partial n} \, \D S = 0 \, . \label{D}
\end{equation}
\end{proposition}

\noindent Here and below $n$ denotes the exterior normal to smooth (of the class
$C^1$) parts of domains' boundaries. The name of this assertion has its origin in
the hydrodynamic interpretation of harmonic functions as velocity potentials
describing irrotational motions of an inviscid fluid in $D \subset \RR^m$, $m =
2,3,$ in the absence of sources in which case the influx is equal to outflux for
every~$D' \subset D$.

\begin{remark}
If $u$ is harmonic in a domain $D \subset \RR^m$, $m \geq 2$, then $u$ is analytic
in $D$ (see \cite{GT}, p.~18), and so $u \in C^\infty (\overline{B})$ for any closed
ball $\overline{B} \subset D$. Moreover, for every $k \geq 1$ we have\\[-3mm]
\begin{equation}
\int_{\partial B} \frac{\partial^k u}{\partial n^k} \, \D S = 0 \, . \label{Dk}
\end{equation}
This follows from Theorem 1 proved in \cite{Ka}, p.~171, and generalizes \eqref{D}
for $D'=B$.
\end{remark}

In 1906, B\^ocher \cite{Bo} and Koebe \cite{Ko} independently discovered the
classical converse to this proposition in two and three dimensions, respectively;
its $n$-dimensional version is as follows.

\begin{theorem}[B\^ocher and Koebe]
Let $D$ be a bounded domain in $\RR^m$, $m \geq 2$. Then $u$ belonging to $C^0
(\overline D) \cap C^1 (D)$ is harmonic in $D$ provided it satisfies the
equality\\[-3mm]
\begin{equation}
\int_{\partial B} \frac{\partial u}{\partial n} \, \D S = 0 \ \ \mbox{for every ball
$B$ such that} \ \overline B \subset D . \label{BK}
\end{equation}
\end{theorem}

\noindent In three dimensions (which does not restrict generality), this theorem is
proved in \cite{K}, pp.~227--228, by deriving equality \eqref{A} from \eqref{BK},
which allows us to apply Theorem~1.4. Along with the latter assertion, the
B\^ocher--Koebe theorem characterizes harmonic functions, but, undeservedly, this
fact is not so widely known now. Indeed, the corresponding references are just
mentioned in three lines in the extensive survey \cite{NV}. Several generalizations
of Theorem~1.5 are described below in Sect.~2.

Another characterization of harmonicity in the two-dimensional case was obtained by
Blasch\-ke in 1916. In the brief note \cite{Bla}, he demonstrated that the property,
now referred to as the asymptotic behaviour of the mean value, is sufficient. The
general form of his assertion is as follows.

\begin{theorem}
Let $D$ be a domain in $\RR^m$, $m \geq 2$. If for $u \in C^0 (D)$ the
equality\\[-3mm]
\begin{equation}
\lim_{r \to +0} r^{-2} [ L (u,x,r) - u (x) ]  = 0 \label{Bl}
\end{equation}
holds for every $x \in D$, then $u$ is harmonic in~$D$.
\end{theorem}

Now we turn to the Kellogg's paper \cite{K1} published in 1934 and opening the line
of works in which the so-called restricted mean value properties are involved. Much
later this notion was defined as follows (see, for example, \cite{HD}).

\begin{definition}
A real-valued function $f$ defined on an open $G \subset \RR^m$ is said to have the
restricted mean value property with respect to balls (spheres) if for each $x \in G$
there exists a ball (sphere) centred at $x$ of radius $r (x)$ such that $B_{r (x)}
(x) \subset G$ and the average of $f$ over this ball (its boundary) is equal to~$f
(x)$.
\end{definition}

\noindent Then Kellogg's result takes the following form.

\begin{theorem}
Let $D$ be a bounded domain in $\RR^m$, $m \geq 2$. If $u \in C^0 (\overline D)$
has~the restricted mean value property with respect to spheres, then $u$ is harmonic
in~$D$.
\end{theorem}

\noindent Further applications of restricted mean value properties are considered in
Sect.~3.

An immediate consequence of formulae \eqref{A} and \eqref{V} is the following
assertion. {\it Let $D$ be a domain in $\RR^m$, $m \geq 2$. If $u$ is harmonic in
$D$, then the equality\\[-3mm]
\begin{equation}
L (u,x,r) = A (u,x,r) \label{L=A}
\end{equation}
holds for all $x \in D$ and all $r > 0$ such that} $B_r (x) \subset D$. In the
two-dimensional case, its converse was obtained by Beckenbach and Reade \cite{BR} in
1943; their simple proof (worth reproducing here) is valid for any $m \geq 2$. The
general formulation is as follows.

\begin{theorem}
Let $D \subset \RR^m$, $m \geq 2$, be a bounded domain. If equality \eqref{L=A}
holds for $u \in C^0 (D)$, all $x \in D$ and all $r$ such that $B_r (x) \subset D$,
then $u$ is harmonic in~$D$.
\end{theorem}

\begin{proof}
If $r \in (0, \rho)$, where $\rho$ is a small positive number, the function $A
(u,x,r)$ is defined for $x$ belonging to an open subset of $D$ depending on the
smallness of $\rho$. Moreover, $A (u,x,r)$ is differentiable with respect to $r$
and\\[-3mm]
\[ \partial_r A (u,x,r) = m r^{-1} [ L (u,x,r) - A (u,x,r) ] = 0 \quad \mbox{for} \ 
r \in (0, \rho) ,
\]
where the last equality is a consequence of \eqref{L=A}. Therefore, $A (u,x,r)$ does
not depend on $r \in (0, \rho)$. On the other hand, shrinking $B_r (x)$ to its
centre by letting $r \to 0$ and taking into account that $u \in C^0 (D)$, one
obtains that $A (u,x,r) \to u (x)$ as $r \to 0$ for $x$ belonging to an arbitrary
closed subset of $D$. Hence for every $x \in D$ we have that $u (x) = A (u,x,r)$ for
all $r \in (0, \rho)$ with some $\rho$, whose smallness depends on $x$. Then
Theorem~1.7 yields that $u$ is harmonic in $D$.
\end{proof}

\noindent Another proof of this result was published by Freitas and Matos \cite{FM},
who, presumably, were unaware of the paper \cite{BR}. However, their paper contains
a generalization of Theorem~1.8 to subharmonic functions; see Gilbarg and
Trudinger's monograph \cite{GT}, p.~13, for their definition.

\begin{definition}
The function $u \in C^2 (D)$ is called subharmonic in a domain $D$ if it satisfies
the inequality $\nabla^2 u \geq 0$ there.
\end{definition}

\noindent A characteristic property of a subharmonic $u \in C^0 (D)$ is as follows.
{\it For every ball $B = B_r (x)$ such that $\overline B \subset D$ and every $h$
harmonic in $B$ and satisfying $h \geq u$ on $\partial B$ the inequality $h \geq u$
holds in $B$ as well.} Therefore, for every subharmonic $u$ we have\\[-3mm]
\begin{equation}
A (u,x,r) \leq L (u,x,r) \quad \mbox{for each} \ B_r (x) \ \mbox{such that both sides
are defined}. \label{L<A}
\end{equation}
It is proved in \cite{FM} that this inequality characterizes subharmonic functions.

\begin{theorem}
Let $u$ be continuous in $D$. Then $u$ is subharmonic in $D$ provided \eqref{L<A}
holds.
\end{theorem}

In conclusion of this section, we mention a property, which is, in some sense,
similar to \eqref{L=A}, and has received much attention; see \cite{NV}, pp.
365--368. It consists in equating the values of $L (u,x,r)$ corresponding to some
$u$ defined on $\RR^m$, $m \geq 2$, for two different radii $r_1, r_2 > 0$ at every
$x$.

The plan of the main part of the paper is as follows. We begin with generalizations
of the B\^ocher--Koebe theorem because these results considered in Sect.~2 are not
so widely known. In Sect.~3, we describe results related to restricted mean value
properties. Mean value properties for non-harmonic functions are considered in
Sect.~4.

\vspace{-3mm}

{\centering \section{Generalizations of the B\^ocher--Koebe theorem}
}

The B\^ocher--Koebe theorem (Theorem 1.5) characterizing harmonic functions in
terms of the zero flux property is not so widely known as various results based on
mean value properties, in particular, restricted ones. In the survey article
\cite{NV}, two sections are devoted to the latter topic, but the authors just
mention a few papers dealing with the zero flux property and its generalizations.
The aim of this section is to fill in this gap at least partially.

However, prior to presenting several results of apparent interest it is worth
noticing that Theorem~1.5 can be improved. It is mentioned after its formulation
that deriving equality \eqref{A} from \eqref{BK} and then applying Theorem~1.4 one
obtains a proof of Theorem~1.5. The assumption made in Theorem~1.4 that \eqref{A}
holds for all spheres in $D$ is superfluous. It can be weakened by using Theorem~1.7
instead of Theorem~1.4, which leads to the following.

\begin{theorem}
Let $D$ be a bounded domain in $\RR^m$, $m \geq 2$. Then $u$ belonging to $C^0
(\overline D) \cap C^1 (D)$ is harmonic in $D$ provided for every $x \in D$ there
exists a radius $r (x)$ such that $\overline B \subset D$, where $B = B_{r (x)}
(x)$, and equality \eqref{BK} holds for this~$B$.
\end{theorem}

\noindent Some extensions of the B\^ocher--Koebe theorem were obtained by Evans
\cite{Ev} (see p.~286 of his paper for the formulation) and Raynor \cite{R} for $m =
2$ and 3, respectively.

\vskip7pt {\bf 2.1. Generalizations of the zero flux property.} If $u \in C^2 (D)$
is a harmonic function in a bounded domain $D \subset \RR^m$, $m \geq 2$, and $v \in
C^1 (D)$, then for any piecewise smooth subdomain $D'$ such that $\overline{D'}
\subset D$ the first Green's formula for $u$ and $v$ is as follows:\\[-3mm]
\begin{equation}
\int_{D'} \nabla u \cdot \nabla v \, \D x = \int_{\partial D'} v \, \frac{\partial
u}{\partial n} \, \D S \, . \label{G}
\end{equation}
It occurs that this equality serves itself as a characteristic of harmonic functions
and yields several other their characteristic properties.

First, it is well known that \eqref{G} defines weak solutions of the Laplace
equation provided $v$ is an arbitrary function from $C^1 (\overline{D'})$ vanishing
on $\partial D'$. It was found long ago that these solutions are harmonic (see, for
example, \cite{Mi}, containing the vast bibliography of classical papers), and in
this sense \eqref{G} characterizes these functions. Second, if $v \equiv 1$, then
\eqref{G} turns into the zero flux property \eqref{D} discussed in Theorems~1.5 and
2.1.

Furthermore, if $v$ is also harmonic, then \eqref{G} implies the equality\\[-3mm]
\begin{equation}
\int_{\partial D'} u \, \frac{\partial v}{\partial n} \, \D S = \int_{\partial D'} v
\, \frac{\partial u}{\partial n} \, \D S \, , \label{Gg}
\end{equation}
which was used by Gergen \cite{G} for obtaining the following generalization of the
B\^ocher--Koebe theorem.

\begin{theorem}
Let $D \subset \RR^m$, $m \geq 2$, be a bounded domain and let $v \in C^2 (\overline
D)$ be a harmonic function in $D$ such that $v > 0$ in $\overline D$. Then $u \in
C^1 (\overline D)$ is harmonic in $D$ provided equality \eqref{Gg} (with $D'$ 
changed to $B$) holds for every ball $B$ such that $\overline B \subset D$.
\end{theorem}

{\bf 2.2. Local characterizations of harmonicity.} In his note \cite{S} published in
1932, Saks improved Theorems 1.5 and 2.2 in the two-dimensional case. The general
form of his first assertion (its simple proof is reproduced here) is as follows.

\begin{theorem}
Let $D \subset \RR^m$, $m \geq 2$, be a bounded domain. Then $u \in C^1 (D)$ is
harmonic in $D$ provided for every $x \in D$\\[-5mm]
\begin{equation}
\lim_{r \to +0} \frac{1}{|S_r| \, r^2} \int_{S_r (x)} \frac{\partial
u}{\partial n} \, \D S = 0 \, , \label{S1}
\end{equation}
where $S_r (x)$ stands for the sphere of radius $r$ centred at $x$ and\/ $|S_r|$ is
its area.
\end{theorem}

\begin{proof}
Let us consider $F (r) = r^{-2} [ L (u,x,r) - u (x) ]$ for arbitrary $x \in D$ and
sufficiently small~$r$. Then\\[-7mm]
\begin{eqnarray*}
&& F (r) = \frac{1}{|S_r| \, r^2} \int_{S_r (x)} [u (y) - u (x)] \, \D S_y
\\ && \ \ \ \ \ \ \ = \frac{1}{|S_r| \, r^2} \int_{S_r (x)} \frac{\partial
u}{\partial n} \left( y + \rho \big( [y-x] / r \big) \frac{x-y}{r} \right) \, \D S_y
\, ,
\end{eqnarray*}
where $0 < \rho \big( [y-x] / r \big) < r$ for all $x \in D$ and all $[y-x] / r \in
S_1$. Since \eqref{S1} implies that $F (r) \to 0$ as $r \to +0$, Theorem~1.6
guarantees that $u$ is harmonic.
\end{proof}

\noindent We omit the formulation of the second theorem by Saks because it
generalizes Theorem~2.2 in the same way as the last theorem generalizes
Theorem~1.5. Instead, we turn to an assertion analogous to the last theorem, but
characterizing biharmonic functions, that is, those which satisfy the equation
\begin{equation}
\Delta^2 u = 0 , \quad \mbox{where} \ \Delta = \nabla^2 . \label{bih}
\end{equation}

\begin{theorem}[Cheng \cite{Ch}]
Let $D$ be a domain in $\RR^2$. Then $u \in C^3 (D)$ is biharmonic in $D$ provided
\begin{equation}
\lim_{r \to +0} \frac{1}{r} \int_0^{2 \pi} \frac{\partial^3 u}{\partial r^3} (x_1 +
r \cos \theta , x_2 + r \sin \theta) \, \D \theta = 0 \label{S2}
\end{equation}
for every $x \in D$.
\end{theorem}

One more characterization of harmonic functions based on an asymptotic property was
obtained by Beckenbach. The property used in his paper \cite{Beck2} is as follows:
\begin{equation}
\int_{B_r (y)} \!\! |\nabla u|^2 \, \D x - \int_{\partial B_r (y)} \!\! u \,
\frac{\partial u}{\partial n} \, \D S = o (r^2) \quad \mbox{as} \ r \to 0 .
\label{Gr}
\end{equation}
Here $y \in D \subset \RR^2$, and the right-hand side expression is obtained from
\eqref{G} by substituting $v = u$ and $D' = B_r (y)$.

\begin{theorem}
Let $u \in C^1 (D)$, and let relation \eqref{Gr} hold for every $y \in D$. If $u$
does not vanish in $D$, then $u$ is harmonic there.
\end{theorem}

\noindent Moreover, it is shown that the assumption $u \neq 0$ in this theorem can
be replaced by the requirement that $\partial^2_{x_1} u$ and $\partial^2_{x_2} u$
are Lebesgue integrable in~$D$.

\vskip7pt {\bf 2.3. Harmonicity via the zero flux property for cubes.} It occurs
that the assertion of Theorem~1.5 remains valid when hyperspheres are replaced by
hypersurfaces bounding $m$-cubes and having edges parallel to the coordinate axes.
Let
\[ Q_r (x) = \{ y \in \RR^m : |y_i - x_i| < r , \ i=1,\dots,m \} , \quad r > 0 ,
\]
denote an open cube centred at $x \in \RR^m$; it is the smallest cube of this kind
containing $B_r (x)$.

\begin{theorem}[Beckenbach \cite{Beck1}]
Let $D$ be a domain in $\RR^m$, $m \geq 2$. If $u$ belongs to $C^1 (D)$ and satisfies
the equality
\begin{equation}
\int_{\partial Q} \frac{\partial u}{\partial n} \, \D S = 0 \label{Beck}
\end{equation}
for every $Q = Q_r (x) \subset D$, then $u$ is harmonic in~$D$.
\end{theorem}

\begin{proof}
Denoting by {\bf 0} the origin in $\RR^m$, we consider the following Steklov-type
mean function
\[ u_r (x) = (2 r)^{-m} \int_{Q_r (\bf 0)} u (x + y) \, \D y \, .
\]
It is defined in some $D_r$ approximating $D$ from inside for small values of $r$.
Then we have
\[ \partial_{x_i}^2 u_r (x) = \frac{1}{(2 r)^m} \int_{-r}^r \!\! \cdots \!\!
\int_{-r}^r \Big[ \partial_{x_i} u (x+y) \big|_{y_i=r} - \partial_{x_i} u (x+y)
\big|_{y_i=-r} \Big] \D_i x \, ,
\]
where $\D_i x = \D x_1 \dots \D x_{i-1} \D x_{i+1} \dots \D x_m$ for $i \neq 1,m$
and $\D_1 x$, $\D_m x$ have appropriate form. Therefore,
\[ \nabla^2 u_r (x) = \frac{1}{(2 r)^m} \int_{\partial Q_r (x)} \frac{\partial u}
{\partial n} \, \D S \, ,
\]
which vanishes on $D_r$ in view of \eqref{Beck}. On each compact subset of $D$, we
have that $u$ is the uniform limit of $u_r$ as $r \to +0$. Since $u_r$ is harmonic
in $D_r$ and this family of domains approximate $D$, we obtain that $u$ is harmonic
in $D$.
\end{proof}

{\centering \section{Results related to restricted mean value properties}
}

{\bf 3.1. The restricted mean value property with respect to spheres combined with
solubility of the Dirichlet problem.} Neither the proof of Theorem 1.7 nor proofs
of related results (see \cite{NV}, Sections 5 and 6, for a review) are trivial.
However, there is a class of bounded domains for which the assertion converse to
this theorem has a very simple proof. Indeed, this takes place when the restricted
mean value property is complemented by the assumption that the Dirichlet problem for
the Laplace equation is soluble in the domain.

\begin{theorem}
Let $u \in C^0 (\overline D)$, where $D \subset \RR^m$, $m \geq 2$, is a bounded
domain in which the Dirichlet problem for the Laplace equation has a solution
belonging to $ C^0 (\overline D)$ for every continuous function given on $\partial
D$. If $u$ has the restricted mean value property in $D$ with respect to spheres,
then $u$ is harmonic in~$D$.
\end{theorem}

\begin{proof}
First, let us show that the theorem's assumptions yield that\\[-5mm]
\begin{equation}
\max_{x \in \overline D} u = \max_{x \in \partial D} u \, . \label{max}
\end{equation}
Denoting the left-hand side by $M$, we notice that it is sufficient to establish
that $u^{-1} (M) \cap \partial D$ is nonempty, where the preimage $u^{-1} (M)$ is a
closed subset of $\overline D$.

Assuming the contrary, we conclude (in view of boundedness of $D$) that there is a
point $x_0 \in u^{-1} (M)$, whose distance from $\partial D$ is minimal and
positive. By the restricted mean value property there exists some $r (x_0) > 0$ such
that $\partial B_{r (x_0)} (x_0) \subset D$ and equality \eqref{A} holds with $r =
r (x_0)$. Since $u (x_0) = M$, this maximal value $u$ attains at every point of
$\partial B_{r (x_0)} (x_0)$, but some of these points is closer to $\partial D$
than $x_0$, which leads to a contradiction proving \eqref{max}.

For $u \in C^0 (\overline D)$ we denote by $f$ its trace on $\partial D$ and by $u_0
\in C^0 (\overline D)$ the solution of the Dirichlet problem in $D$ with $f$ as the
boundary data. Hence $u_0$ has the unrestricted mean value property in $D$ with
respect to spheres, and so \eqref{max} holds for $u - u_0$ as well as for $- (u -
u_0)$, which implies that $u \equiv u_0$ in $D$ because $u \equiv u_0$ on $\partial
D$. Thus, $u$ is harmonic in~$D$.
\end{proof}

\noindent In the brief note \cite{B} by Burckel, this theorem is proved for
two-dimensional domains, whereas simple one-dimensional examples demonstrating that
the assumptions about boundedness of $D$ and continuity of $u_0$ in $\overline D$
are essential for validity of the theorem are given in \cite{CH}, pp.~280--281.

The question how to describe domains in which the Dirichlet problem is soluble has a
long history going back to the \emph{Essay on the Application of Mathematical
Analysis to the Theories of Electricity and Magnetism} by George Green (published in
1828), where this problem was posed for the first time. In the 19th century, the
well-known results about this problem were obtained by Gauss, Dirichlet, Riemann,
Weierstrass, C.~Neumann, Poincar\'e, Lyapunov and Hilbert. The final answer to the
above question was given by Wiener \cite{NW} in 1924, who introduced the notion of
capacity for this purpose. A detailed review of his result as well as of the
preceding work accomplished during the first quarter of the 20th century one finds
in the Kellogg's article \cite{Ke}.

\vskip7pt {\bf 3.2. An example due to Littlewood.} It had been found rather long ago
that the re\-stricted mean value property with respect to balls does not guarantee
harmonicity of a $C^0 (D)$-function in a bounded domain $D$ without some extra
bounding assumption. There are several examples demonstrating this; see \cite{NV},
p.~369 for references. Here, we reproduce the example proposed by Littlewood and
published in Huckemann's paper \cite{Hu}, p.~492 (in \cite{NV}, this example is
mentioned in passing on p.~369, in a quotation from the Littlewood's booklet
\cite{Lit}).

Let $D = B_1 ({\bf 0}) \subset \RR^2$ and let us define $u$ on $D$ by
putting\\[-3mm]
\begin{equation}
u (x) = a_k \log |x| + b_k \ \mbox{for} \ |x| \in [1-2^{-k} , 1-2^{-k-1} ] , \quad k
= 0,1,2,\dots , \label{lit}
\end{equation}
where $a_k$ and $b_k$ are chosen so that $u \in C^0 (D)$ (in particular, this means
that $a_0 = 0$) and $u$ satisfies the restricted mean value property with respect to
discs for all points on each circumference $|x| = 1-2^{-k}$ (it is obvious that this
property holds elsewhere). For the latter purpose it is sufficient to require that
$a_1, a_2, \dots$ have alternating signs and $|a_k|$ grows sufficiently fast with
$k$. It is clear that $u$ defined by \eqref{lit} is not harmonic in $D$.

\vskip7pt {\bf 3.3. On harmonicity of harmonically dominated functions.} Presumably,
the brief note \cite{V1} by Veech was the first publication, in which condition (B) was
used together with the assumption that the absolute value of the function under
consideration is majorized by a positive harmonic function. The result announced in
\cite{V1} (its improved version was proved in \cite{V2}) we formulate keeping the
original notation.

\begin{theorem}
Let $\Omega$ be a bounded Lipschitz domain in the plane, and let $f$ be a Lebesgue
measurable function on $\Omega$ such that $|f(x)| \leq g(x)$, $x \in \Omega$, for
some positive harmonic function $g$ on $\Omega$. If for each $x \in \Omega$ there is
a disc contained in $\Omega$ and centered at $x$ over which the average of $f$ is $f
(x)$, and if $\delta (x)$, the radius of this disc, as a function of $x$ is bounded
away from $0$ on compact subsets of\/ $\Omega$, then $f$ is harmonic.
\end{theorem}

It is clear straight out of the title of \cite{V2} that the proof of this theorem
given by Veech relies heavily on probabalistic methods. Purely analytic proof of an
analogous result was obtained by Hansen and Nadirashvili in their seminal article
\cite{HN}. To outline their approach we begin with describing the required notation
and definition.

In what follows, $h$ is a fixed harmonic function on a bounded domain $D \subset
\RR^m$, for which the inequality $h \geq 1$ holds. A function $f$ on $D$ is called
$h$-bounded if there exists a constant $c > 0$ such that $|f| \leq c h$ throughout
$D$.

\begin{definition}
Let a positive function $r$ on $D$ be such that $r (x) \leq \rho (x)$, where $\rho
(x)$ is the distance from a point $x \in D$ to $\partial D$. A Lebesgue measurable
function $f$, which is $h$-bounded in $D$, is said to be $r$-median
provided\\[-3mm]
\[ f (y) = \frac{1}{|B_{r (y)} (y)|} \int_{B_{r (y)} (y)} \!\! f (x) \, \D x \quad 
\mbox{for every} \ x \in D .
\]
As in \eqref{S1}, by $|\cdot|$ the Lebesgue measure of the corresponding set is
denoted.
\end{definition}

\noindent Now we are in a position to formulate results proved in \cite{HN}.

\begin{theorem}
Let $D$ be a bounded domain in $\RR^m$, $m \geq 2$, in which a harmonic function $h
\geq 1$ is given, and let $r$ be a function described in Definition 3.1, then the
following two assertions are~true.

{\rm (i)} If $u \in C^0 (D)$ is an $h$-bounded and $r$-median function, then $u$ is
harmonic in $D$.

{\rm (ii)} If $u$ is a Lebesgue measurable, $h$-bounded and $r$-median function in
$D$, then $u$ is harmonic 

\hskip16pt there provided $r$ is bounded away from $0$ on every compact subset of
$D$.
\end{theorem}

\noindent Assertion (i) of this theorem establishes, in particular, that the
restricted mean value property implies harmonicity for bounded continuous functions,
whereas assertion (ii) extends this fact to Lebesgue measurable functions at the
expense of an extra assumption imposed on diameters of balls in Definition~3.1. The
latter imposes a geometrical restriction on the domain $D$.

 To
give an idea how complicated is the proof of Theorem~3.3 it is sufficient to list
some of different conceptions used by Hansen and Nadirashvili in their
considerations: (1)~the Martin compactification; (2) the Schr\"odinger equation with
singularity at the boundary (it is investigated in \cite{HN}, Sect.~2; (3)
transfinite sweeping of measures (it is studied in \cite{HN}, Sect.~3). In
\cite{HN2}, the assertions of Theorem~3.3 are extended to the case of more general
mean value properties and to domains which are not necessarily bounded; namely, it
is required that $D \neq \RR^m$ for $m \geq 3$, whereas the complement of $D$ is a
non-polar set when $m=2$ (see \cite{H}, Ch.~7, Sect.~1, for the definition of a
polar set).

\vskip7pt {\bf 3.4. On two conjectures related to restricted mean value
properties.} In his booklet \cite{Lit} published in 1968, Littlewood posed several
questions among which was the following one (it is usually referred to as the one
circle problem). Let $D = B_1 ({\bf 0}) \subset \RR^2$ and let $u \in C^0 (D)$ be
bounded on $D$. Is $u$ harmonic if for every $x \in D$ there exists $r (x) \in (\, 0
, 1-|x| \,]$ such that the equality
\begin{equation}
u (x) = \frac{1}{2 \pi r (x)} \int_{\partial B_{r (x)} (x)} \!\! u \, \D S \ \
\mbox{holds ?} \label{oc}
\end{equation}
Littlewood conjectured that the answer to this question is `No' and this was
established by Hansen and Nadirashvili \cite{HN3}, who proved the following
assertion.

\begin{theorem}
Let $D = B_1 ({\bf 0}) \subset \RR^2$ and let $\alpha \in (0, 1)$. Then there exists
$u \in C^0 (D)$ which attains values in $[0, 1]$ and for every $x \in D$ satisfies
equality \eqref{oc} with some $r (x)$ belonging to $(\, 0 , \alpha [\, 1-|x| \,]
\,)$, but is not harmonic in~$D$.
\end{theorem}

\noindent In fact, this result is a corollary of another theorem in which $u$ is
averaged not over a circumference as it takes place in formula \eqref{oc}, but over
an annulus centred at $x$ and enclosed in $D$. It is also worth mentioning that a
certain random walk is applied for describing the function, whose existence is
asserted in this theorem. The construction `is very delicate' as is emphasized in
the subsequent paper \cite{Ha1}, where it is substantially simplified. For this
purpose a result due to Talagrand \cite{Ta} is used, it concerns the Lebesgue
measure of projections of a certain two-dimensional set on a straight line.

Let $\alpha \in [-\pi /2, \pi /2]$, by $p_\alpha : \RR^2 \to \RR$ the projection
operator is denoted mapping onto the $x_1$-axis parallel to the line going through
the origin and forming the angle $\alpha$ with this axis. The following assertion
(it is used in \cite{Ha1} to simplify the proof of Theorem~3.4) is a special case of
Theorem~1 proved in \cite{Ta}.

\begin{proposition}
Let $a, b \in \RR$, $a < b$. Then there exists a compact $K \subset [0, 1] \times
[a, b]$ such that the orthogonal projection of $K$ on the $x_1$-axis coincides with
$[0, 1]$, whereas the Lebesgue measure of $p_\alpha (K) \subset \RR$ is equal to
zero for every $\alpha \in (-\pi /2, \pi /2)$.
\end{proposition}

The second conjecture was formulated by Veech \cite{V3} in 1975. It involves the
notion of ad\-missible function on a domain $D \subset \RR^m$, $m \geq 2$, by which
a positive function $r$ is understood such that $B_{r (x)} (x) \subset D$ for every
$x \in D$.

\begin{conjecture}
Let $D$ be a bounded domain in $\RR^m$ and let $r$ be an admissible function on $D$
which is locally bounded away from zero, that is, it satisfies the inequality
$\inf_{x \in K} r (x) > 0$ for each compact $K \subset D$. Then every non-negative,
$r$-median function on $D$ is harmonic.
\end{conjecture}

\noindent Huckemann \cite{Hu} demonstrated that an analogous, one-dimensional
assertion is true. However, there are measurable and continuous counterexamples to
Conjecture~3.1 for $m \geq 2$. Like that considered in Theorem~3.4, these examples
are based on properties of the random walk given by a certain transition kernel. We
restrict ourselves to formulating the following assertion similar to Theorem~3.4;
see \cite{HN4}.

\begin{theorem}
Let $D = B_1 ({\bf 0}) \subset \RR^m$, $m \geq 2$, and let $\alpha \in (0, 1]$. Then
there exist strictly positive functions $u$ and $r$ belonging to $C^0 (D)$ such that
$r (x) \leq \alpha (1-|x|)$ for all $x \in D$ and $u$ is $r$-median, but not
harmonic in~$D$.
\end{theorem}

{\bf 3.5. The restricted mean value property for circumferences and Liouville's
theorem in two dimensions.} In the Mikhlin's textbook \cite{M} (see Ch.~12,
Sect.~4), the proof of the general Liouville's theorem, asserting that {\it a
harmonic function defined throughout $\RR^m$ and bounded above (or below) is
constant}, is based on Theorem~1.1 (the mean value property for spheres). However,
it occurs that the assumption about harmonicity in Liouville's theorem is
superfluous at least in the two-dimensional case. The following assertion shows that
it is sufficient to require the restricted mean value property for circumferences.

\begin{theorem}
Let a real-valued function $u \in C^0 (\RR^2)$ be bounded. If there exist a strictly
positive function $r$ on $\RR^2$ and a constant $M > 0$ such that $r (x) \leq |x| +
M$ whenever $|x| \geq M$, then $u$ is constant provided the equality\\[-3mm]
\begin{equation}
u (x) = \frac{1}{2 \pi} \int_0^{2 \pi} u \big( x + r (x) \, \E^{\ii t} \big) \, \D
t \label{3.6}
\end{equation}
holds for every $x \in \RR^2$.
\end{theorem}

\noindent The original proof of this theorem published by Hansen \cite{Ha2} relies
on `a rather technical minimum principle involving the Choquet boundary of a compact
set with respect to a function cone' as is pointed out in the subsequent paper
\cite{Ha3}. The latter contains a new proof which involves only elementary geometry
and basic facts like the inequality $\log (1 + a) \leq a$ for $a > 0$ and the mean
value property
\[ (2 \pi)^{-1} \int_0^{2 \pi} \log |x + \rho \E^{\ii t}| \, \D t = \log |x| \quad
\mbox{for} \ \rho \in (0, |x|) , \ x \in \RR^2 .
\]
Some other results concerning Liouville’s theorem and the restricted mean volume
property in the plane can be found in the brief note \cite{Ha4}.

\vskip7pt {\bf 3.6. On the one ball problem in $\mathbf{\RR^m}$, $\mathbf{m \geq
2}$.} The result presented in this section is an improvement of a theorem obtained
by Flatto \cite{Fl}. It concerns the following question which to some extent is
similar to the Littlewood's one circle problem considered in Sect.~3.4. Let $u \in
C^0 (\RR^m)$, $m \geq 2$, and for a certain fixed $r > 0$ the mean value equality
\begin{equation}
u (x) = \frac{1}{\omega_m r^{m}} \int_{|y| < r} u (x + y) \, \D y \label{Vv}
\end{equation}
holds for all $x \in \RR^m$. What growth condition must be imposed on $u (x)$ as
$|x| \to \infty$ to guarantee $u$ be harmonic? The answer involves properties of
zeros of the even, entire function
\[ \eta (z) = 2^{m/2} \Gamma \left( \frac{m}{2} + 1 \right) \frac{J_{m/2} (z)}{z^{m/2}}
- 1 \, ,
\]
where $J_\nu$ is the $\nu$th Bessel function. It occurs (see \cite{Vol}, Sect.~3)
that along with the double zero at the origin this function has a sequence $\{ z_k
\}_1^\infty$ of simple zeros such that $\Re z_k > 0$ and $|\Im z_k| > 0$ for all $k
\geq 1$. Moreover, there exists $\mu = \min_{k \geq 1} |\Im z_k| > 0$, which allows
us to define $h (x, r) = |x|^{(1-m)/2} \exp \{\mu |x| / r \}$, which plays the
crucial role in following assertion proved by Volchkov \cite{Vol} in 1994.

\begin{theorem}
Let $u \in C^0 (\RR^m)$, $m \geq 2$, satisfy \eqref{Vv} for all $x \in \RR^m$ and
some fixed $r > 0$. Then $u$ is harmonic provided $u (x) = o (h (x, r))$ as $|x| \to
\infty$.

On the other hand, there exists a function $u \in C^\infty (\RR^m)$ and $r > 0$ such
that \eqref{Vv} holds for all $x \in \RR^m$ and $u (x) = O (h (x, r))$ as $|x| \to
\infty$, but $u$ is not harmonic.
\end{theorem}

\noindent Another theorem obtained in \cite{Vol} deals with the case when the mean
value over a second ball is involved through the operator
\[ (B u) (x) = u (x) - \frac{1}{\omega_m r_1^{m}} \int_{|y| < r_1} u (x + y) \, \D y
\, .
\]
Let $A$ denote the set of quotients each having some elements of $\{ z_k
\}_1^\infty$ as the numerator and denominator.

\begin{theorem}
Let $u \in C^0 (\RR^m)$, $m \geq 2$, satisfy \eqref{Vv} for all $x \in \RR^m$ and
some $r > 0$. Then $u$ is harmonic provided $r/r_1 \notin A$ and $(B u) (x) = o (h
(x, r))$ as $|x| \to \infty$.

On the other hand, for any $r, r_1 > 0$ there exists a function $u \in C^\infty
(\RR^m)$ such that \eqref{Vv} holds for all $x \in \RR^m$ and $(B u) (x) = O (h (x,
r))$ as $|x| \to \infty$, but $u$ is not harmonic.
\end{theorem}

\noindent Similar results are true when the volume mean values are changed to the
area ones. Furthermore, `local' versions (that is, for $u$ given in a ball instead
of $\RR^m$) of these theorems are proved in \cite{Vol1}.

{\centering \section{Mean value properties for non-harmonic functions}
}

Since Netuka and Vesel\'y \cite{NV} had considered exclusively harmonic functions
and solutions of the heat equation (see also a comprehensive treatment of the heat
potential theory in the monograph \cite{Wa} by Watson), mean value theorems and related
results for some other partial differential equations are presented in this section.

\vskip7pt {\bf 4.1. Biharmonic functions.} Presumably, the first generalization of
the Gauss-type mean value formula \eqref{V} for higher order elliptic equations was
obtained by Pizzetti. In 1909, he con\-sidered polyharmonic functions, that is, $C^{2
k}$-functions, $k=2,3,\dots$, which satisfy the equation $\Delta^k u = 0$; see the
original note \cite{Piz}, whereas the three-dimensional version of his formula is
given in \cite{CH}, p.~288. A description of the general form of this formula and
certain its generalizations can be found in \cite{BP}. To give an idea of Pizzetti's
results we restrict ourselves to the case of biharmonic mean formulae valid for
functions which satisfy equation \eqref{bih} in a domain $D \subset \RR^m$. The
first formula
\begin{equation}
u (y) = \frac{1}{\omega_m R^{m}} \int_{B_R (y)} \!\! u \, \D x - \frac{R^2}{2 (m+2)}
\nabla^2 u (y) \, , \label{bmv}
\end{equation}
involving the mean value over an arbitrary ball $B_R (y) \subset D$ and analogous to
\eqref{V}, can be found in the classical book \cite{Ni} by Nicolescu. The second
formula is as follows:
\begin{equation}
u (x) = \frac{1}{|\partial B|} \int_{\partial B} \!\! u \, \D S - \frac{r^2}{2 m}
\nabla^2 u (x) \, . \label{bma}
\end{equation}
Here $B = B_r (x)$ is an arbitrary ball in $D$ and $|\partial B|$ is the area of its
boundary. A simple deriva\-tion of \eqref{bma} is given in \cite{KH}; it uses
Green's function for the Laplace equation in a ball and the explicit form of this
function is well-known in the case of the Dirichlet condition on $\partial B$.

Let us illustrate how properties of biharmonic functions analogous to those of
harmonic ones follow from, say \eqref{bmv}. A simple example is Liouville's theorem.

\begin{theorem}
Let $u$ be a bounded biharmonic function on $\RR^m$, then $u \equiv
\mathrm{const}$.
\end{theorem}

\begin{proof}
If $\sup_{x \in \RR^m} |u (x)| = M < +\infty$, then \eqref{bmv} implies that
\[ \sup_{x \in \RR^m} |\nabla^2 u (x)| \leq \frac{4 (m+2)}{R^2} M \, ,
\]
where $R > 0$ is arbitrary. Hence $\nabla^2 u$ vanishes identically on $\RR^m$, and
so $u \equiv \mathrm{const}$ by Liouville's theorem for harmonic functions; see
Sect.~3.5.
\end{proof}

As another example of similarity between properties of biharmonic and harmonic
functions we consider the equality analogous to \eqref{Dk}. It has the same form
\[ \int_{\partial B} \frac{\partial^k u}{\partial n^k} \, \D S = 0 
\]
with $\overline{B}$ being an arbitrary closed ball in a domain $D \subset \RR^m$, $m
\geq 2$, where $u$ is biharmonic. What distinguishes the last formula from
\eqref{Dk} is that $k \geq 3$ here, whereas $k \geq 1$ in \eqref{Dk}. Again, the
last formula follows from Theorem 1 proved in \cite{Ka}, p.~171.

\vskip7pt {\bf 4.2. The restricted mean value property (4.1) and Liouville's
theorem.} According to Theorem~3.6, it is sufficient to require the restricted mean
value property \eqref{3.6} instead of harmonicity in order to guarantee that a
bounded function is constant. It occurs that the same is true if one imposes
condition \eqref{bmv} instead of \eqref{3.6}. This follows from the assertion (see
\cite{EK} for the proof), which differs from Theorem~3.6 in two ways: another
restricted mean value property is used and the high-dimensional Euclidean space is
considered instead of the plane.

\begin{theorem}
Let a real-valued, Lebesgue measurable function $u$ be bounded on $\RR^m$, $m \geq
3$, and let its Laplacian in the distribution sense be a bounded function. If there
exist a strictly positive function $R$ on $\RR^m$ and a constant $M > 0$ such that
$R (y) \leq |y| + M$, then $u$ is constant provided equality \eqref{bmv} with $R = R
(y)$ holds for every $y \in \RR^m$ and either $u \in C^0 (\RR^m)$ or $R$ is locally
bounded from below by a positive constant.
\end{theorem}

\noindent There is another result based on equality \eqref{bmv} in the note
\cite{EK}. It is aimed at proving harmonicity of a locally Lebesgue integrable
function which is harmonically dominated in a domain lying in $\RR^m$, $m \geq 3$.

{\bf 4.3. Koebe-type and Harnack-type results for biharmonic functions.} Like in
the case of harmonic functions, both \eqref{bmv} and \eqref{bma} guarantee that a
function $u$ is biharmonic provided these equalities hold for all admissible balls
(that is, lying within a domain $D$) centred at almost every point of~$D$. Namely,
the following assertion was proved in \cite{KH} (cf. \cite{NV}, Theorem~3.1).

\begin{theorem}
Let $u$ be a locally Lebesgue integrable function on a domain $D \subset \RR^m$, $m
\geq 2$, and let its Laplacian in the distribution sense be a function. Then the
following conditions are equivalent:

{\rm (i)} the function $u$ is biharmonic on $D$;

{\rm (ii)} for almost every $y \in D$ equality \eqref{bmv} holds for all $R > 0$
such that $\overline{B_R (y)} \subset D$;

{\rm (iii)} for almost every $y \in D$ equality \eqref{bma} holds for all $R > 0$
such that $\overline{B_R (y)} \subset D$.
\end{theorem}

\noindent A natural consequence of this theorem is the following assertion
analogous to Corollary 1.4.

\begin{corollary}[The Harnack-type convergence theorem]
Let every function of a sequence $\{u_k\}$ be biharmonic in $D$. If the convergence
$u_k \to u$ as $k \to \infty$ is locally uniform in $D$, then $u$ is biharmonic in $D$.
\end{corollary}

Furthermore, Bramble and Payne \cite{BP} obtained direct and converse mean value
properties for polyharmonic functions in terms different from those used by Pizzetti.
As above, we restrict ourselves to biharmonic functions only in formulations of
these properties.

\begin{theorem}
Let $u$ be a biharmonic function in a domain $D \subset \RR^m$, $m \geq 2$. If $r_1
< r_2$ are positive numbers and $B_{r_2} (y) \subset D$ for some $y \in D$, then
\begin{equation}
u (y) = \left[ \frac{r_2^2}{\omega_m r_1^{m}} \int_{B_{r_1} (y)} \!\! u \, \D x -
\frac{r_1^2}{\omega_m r_2^{m}} \int_{B_{r_2} (y)} \!\! u \, \D x \right] \bigg/
\left( r_2^2 - r_1^2 \right) \, . \label{bV}
\end{equation}
\end{theorem}

\noindent The assertion with volume means changed to spherical ones in the square
brackets is also true, whereas the converse is as follows.

\begin{theorem}
Let $u$ be a locally Lebesgue integrable function on a domain $D \subset \RR^m$, $m
\geq 2$. If $u$ satisfies \eqref{bV} for almost every $y \in D$ and all positive
$r_1 < r_2$ with sufficiently small $r_2$, then $u$ is equal almost everywhere in $D$
to a biharmonic function.
\end{theorem}

\noindent Analogues of Theorems 4.4 and 4.5 for polyharmonic functions involve
certain determinants in the numerator and denominator of the fraction on the
right-hand side of \eqref{bV}.

\vskip7pt {\bf 4.4. Metaharmonic functions in a domain.} This term serves as a
convenient (though, may be, an out-of-use) abbreviation for `solutions of the
Helmholtz equation', like the term `harmonic functions' is a widely used equivalent
to `solutions of the Laplace equation'. Indeed, the Helmholtz equation
\begin{equation}
\nabla^2 u + \lambda^2 u = 0 , \quad \mbox{where} \ \lambda \in \CC , \label{Hh}
\end{equation}
has the next level of complexity comparing with the Laplace equation $\nabla^2 u =
0$, and so it is reasonable to use the Greek prefix {\it meta-} (equivalent to Latin
{\it post-}) in order to denominate solutions of \eqref{Hh}. Presumably, the term
metaharmonic functions was introduced by I.~N. Vekua in his still widely cited
article \cite{Ve1} (its English translation was published as Appendix~2 in \cite{Ve2}
and is available online at:
ftp://ftp.math.ethz.ch/hg/EMIS/journals/TICMI/lnt/vol14/vol14.pdf).

Equation \eqref{Hh} was briefly considered by Euler and Lagrange in their studies of
sound propagation and vibrating membranes as early as 1759, but it was Helmholtz who
initiated detailed investigation of this equation now named after him. The aim of
his article \cite{Helm} published in 1860 was to describe sound waves in a tube with
one open end (organ pipe), for which purpose he derived a representation of
solutions to \eqref{Hh} analogues to the Green's representation formula for harmonic
functions. Thus, the way was opened to obtaining mean value properties for
metaharmonic functions, and this was realised by Weber in his papers \cite{W1} and
\cite{W2}, in which the following formulae similar to \eqref{A}
\begin{equation}
u (x) = \frac{\lambda r}{4 \pi r^2 \sin \lambda r} \int_{\partial B_r (x)} \!\! u \,
\D S \quad \mbox{and} \quad u (x) = \frac{1}{2 \pi r J_0 (\lambda r)} \int_{\partial
B_r (x)} \!\! u \, \D S \, , \ \ \lambda > 0 , \label{We}
\end{equation}
were found in three- and two-dimensional cases, respectively (see also \cite{CH},
pp.~288 and 289, respectively); here $J_0$ is the Bessel function of order zero. The
counterparts of formulae \eqref{We} for the so-called modified Helmholtz equation
(in which $\lambda = \pm \ii \kappa$ with $\kappa > 0$) are given in \cite{Po},
where a generalization to solutions of $(\nabla^2 - \kappa^2)^p u = 0$,
$p=2,3,\dots$, is considered. The general mean value property for spheres is
derived, for example, in \cite{CH}, p.~289.

\begin{theorem}
Let $D$ be a domain in $\RR^m$, $m \geq 2$. If $u \in C^2 (D)$ satisfies equation
\eqref{Hh} in $D$, then
\begin{equation}
u (x) = \frac{N (m, \lambda)}{m \, \omega_m r^{m-1}} \int_{\partial B} u \, \D S \,
, \quad N (m, \lambda) = \frac{(\lambda r / 2)^{(m-2)/2}}{\Gamma (m/2) \, J_{(m-2)/2}
(\lambda r)} \, , \label{Ah}
\end{equation}
for every ball $B = B_r (x)$ such that $\overline B \subset D$; $J_\nu$ is the
$\nu$th Bessel function.
\end{theorem}

\noindent It is straightforward to calculate that
\begin{equation}
\frac{N (m, \lambda)}{m \, \omega_m r^{m-1}} = \frac{1}{\lambda J_{(m-2)/2} (\lambda
r)} \left( \frac{\lambda}{2 \pi r} \right)^{m/2} . \label{levit}
\end{equation}
A new approach to the derivation of \eqref{Ah} was developed in the note \cite{GN},
which also contains the following converse to Theorem~4.6.

\begin{theorem}
Let $D$ be a bounded domain in $\RR^m$, $m \geq 2$. If $u \in C^0 (D)$ and for every
$x \in D$ there exits $r_* (x) > 0$ such that $B_{r_* (x)} \subset D$ and equality
\eqref{Ah} holds for every $B = B_r (x)$ with $r < r_* (x)$, then $u$ is
metaharmonic in~$D$.
\end{theorem}

\noindent The proof involves the mollifier technique used in \cite{M} for proving
the Koebe theorem.

\vskip7pt {\bf 4.5. Metaharmonic functions on $\mathbf{\RR^m}$, $\mathbf{m \geq
2}$.} In view of self-similarity, it is sufficient to consider solutions of\\[-5mm]
\begin{equation}
\nabla^2 u + u = 0 ,  \label{H1}
\end{equation}
in which case an assertion analogous to Theorems 4.6 and 4.7 is as follows.

\begin{theorem}
A function $u \in C^0 (\RR^m)$, $m \geq 2$, satisfying \eqref{H1} in the
distribution sense, is a solution of this equation if and only if the
equality\\[-3mm]
\begin{equation}
(2 \pi r)^{m/2} J_{m/2} (r) \, u (x) = \int_{|y| < r} \!\! u (x+y) \, \D y 
\label{Rh}
\end{equation}
holds for every $r > 0$ and every $x \in \RR^m$. Hence the integral vanishes when
$r$ is equal to any positive zero of $J_{m/2}$.
\end{theorem}

\begin{remark}
It is remarkable that the expressions involved in \eqref{levit} and \eqref{Rh} are
also used in an estimate of the spectral function for the Dirichlet Laplacian.
Indeed, dividing $J_{m/2} (\lambda r)$ by $(2 \pi r / \lambda)^{m/2}$ one obtains
the principal term in the estimate (0.6) proved in \cite{Lev} (see p.~268 of this
paper). Of course, the meaning of $\lambda$ and $r$ in the latter estimate differs
from that in formulae \eqref{levit} and \eqref{Rh} (notice that $\lambda = 1$ in the
last formula).
\end{remark}

\noindent Below, the standard notation is used for the $k$th positive zero of
$J_{\nu}$, namely, $j_{\nu, k}$, $k=1,2,\dots$. In his note \cite{Vol2} published in
1994, Volchkov proved a converse to the last assertion of Theorem~4.8; it involves
the sequence of functions\\[-3mm]
\begin{equation}
\Phi_k (x) = \int_{|y| < j_{m/2, k}} \!\! u (x+y) \, \D y \label{Phi}
\end{equation}
defined with the help of zeros of $J_{m/2}$.

\begin{theorem}
If $u \in L^1_{\rm loc} (\RR^m)$, $m \geq 2$, is such that $\Phi_k$ vanishes
identically on $\RR^m$ for all $k=1,2,\dots$, then $u$ coincides almost everywhere
with a solution of equation \eqref{H1}.
\end{theorem}

\noindent Furthermore, it occurs that if mean values of $u$ over all balls of one
particular radius $j_{m/2, k}$ vanish, then $u$ is metaharmonic provided all its
$L^2$-type characteristics\\[-3mm]
\[ M_{r, k} (u) = \int_{|x| < r} \!\! |\Phi_k (x)|^2 \, \D x \, , \quad k=1,2,\dots ,
\]
grow not not too fast.

\begin{theorem}
Let $u \in L^1_{\rm loc} (\RR^m)$, $m \geq 2$, be such that $\Phi_k$ vanishes
identically on $\RR^m$ for some fixed $k$, then $u$ coincides almost everywhere with
a solution of equation \eqref{H1} provided $M_{r, k} (u) = o (r)$ as $r \to \infty$
for all $k=1,2,\dots$.

On the other hand, there exists a function $u \in C^\infty (\RR^m)$ such that
$\Phi_k$ vanishes identically on $\RR^m$ for some fixed $k$ and $M_{r, k} (u) = O
(r)$ as $r \to \infty$ for all $k=1,2,\dots$, but $u$ is not metaharmonic.
\end{theorem}

\noindent The method used for proving Theorems~4.9 and 4.10 in \cite{Vol2} is
applicable to mean values over spheres, thus allowing to obtain similar results in
this case.

\vskip7pt {\bf 4.6. Metaharmonic functions on infinite domains in $\mathbf{\RR^m}$,
$\mathbf{m \geq 2}$.} To prove an asser\-tion similar to Theorem~4.9 is much more
complicated task in the case when $D$ is an unbounded domain not coinciding with
$\RR^m$. The reason is that solutions of \eqref{H1} inevitably loose (at least
partly) the translation invariance in such a domain. The first result of this kind
concerns domains of the form $\RR^m \setminus K$, where $K$ is a convex compact set,
and its proof requires essentially new methods; see \cite{Vol3}.

In the recent paper \cite{O}, a wide class of domains, say $\mathcal{O}$, was
introduced and each domain $D \in \mathcal{O}$ has the following properties:

(a) it contains the half-space $H = \{ x \in \RR^m : x_m > 0 \}$, 

(b) for every $x \in D$ there exists a point $x_*$ such that $x \in
\overline{B_{j_{m/2, 1}} (x_*)} \subset D$, 

(c) the set of all admissible centres $x_*$ is a connected subset of $D$.

\noindent It is clear that any half-space belonging to $D$ takes the form $H$ after
an appropriate change of variables involving translation and rotation within~$D$.
Furthermore, for every $k=1,2,\dots$ the function $\Phi_k$ given by \eqref{Phi} is
defined on its own subdomain $D_k \subset D$.

\begin{theorem}
Let a domain $D \subset \RR^m$, $m \geq 2$, belong to the class $\mathcal{O}$, and
let $u \in L^1_{\rm loc} (D)$ be such that\\[-5mm]
\begin{equation}
\int_\alpha^\beta \int_{\RR^{m-1}} |u (x)| \, \D x_1 \cdots \D x_{m-1} \D x_m <
+\infty \label{last'}
\end{equation}
for any positive $\alpha < \beta$. Then the following two assertions are
equivalent:

$\bullet$ $\Phi_k$ vanishes identically on $D_k$ for all $k=1,2,\dots;$

$\bullet$ $u$ coincides almost everywhere in $D$ with a solution of equation
\eqref{H1}.
\end{theorem}

\noindent From the proof given in \cite{O}, Sect.~3, it follows that both conditions
(b) and (c) are required to describe the class $\mathcal{O}$ used in this theorem.
Indeed, the assertion is not true if either of these conditions is omitted; see
considerations at the end of the cited section. The next theorem demonstrates in
what sense \eqref{last'} is essential for guaranteeing that equation \eqref{H1} holds
along with vanishing of all $\Phi_k$, $k=1,2,\dots$.

\begin{theorem}
For every $\delta > 0$ there exists a sequence $\{ u_k \} \subset C^\infty (H)$ such
that $u_i$ and~$u_j$ are not equal identically for $i \neq j$, and the following
properties are fulfilled simultaneously:

$\bullet$  $\int_H |u_k (x)| \, \E^{\delta x_m} \, \D x < + \infty$ for all
$k=1,2,\dots ;$

$\bullet$ each function $\int_{|y| < j_{m/2, k}} \!\! u_k (x+y) \, \D y$ vanishes
identically on the corresponding subset of $H ;$

$\bullet$ for every $k=1,2,\dots$ the function $u_k$ does not satisfy equation
\eqref{H1} in $H$.
\end{theorem}

\noindent Thus, these theorems provide a definitive result in the case of a
half-space.

\vskip7pt {\bf 4.6. The mean value property over discs in $\mathbf{\RR^2}$ and
metaharmonic functions.} According to the approach developed by Chamberland in his
note \cite{Cham}, the two-dimensional mean value formula\\[-3mm]
\begin{equation}
u (x) = \frac{1}{\pi R^{2}} \int_{B_R (x)} u (y) \, \D y , \quad x \in \RR^2 ,
\label{cha}
\end{equation}
where a constant $R > 0$ is fixed, is considered as equation for unknown $u \in C^0
(\RR^2)$. It is clear that the space of these solutions is translation-invariant,
rotation-invariant and closed in the usual topology. It occurs that these properties
imply the following.

\begin{theorem}
Every $u \in C^0 (\RR^2)$ satisfying \eqref{cha} belongs to the closed linear span
of the solutions of equation \eqref{Hh}, where $\lambda R = 2 J_1 (\lambda R)$ and
$J_1$ is the Bessel function of order one.
\end{theorem}

\noindent This assertion is proved by using the spectral synthesis technique.

\vspace{2mm}

{\bf Acknowledgement.} The author is indebted to O. V. Motygin for help with
bibliography.

\renewcommand{\refname}{
\begin{center}{\Large\bf References}
\end{center}}
\makeatletter
\renewcommand{\@biblabel}[1]{#1.\hfill}
\makeatother

\end{document}